\documentclass[12pt]{article}

\usepackage{amsmath,amssymb,amsthm,graphicx,color}
\usepackage{amsfonts}
\usepackage{verbatim}
\usepackage[left=1in,right=1in,top=1in,bottom=1in]{geometry}
\usepackage{cite,multirow}
\usepackage[british]{babel}
\usepackage{enumitem}
\usepackage{hyperref} 

\hypersetup{
    colorlinks=false,
    pdfborder={0 0 0},
}

\newtheorem{theorem}{Theorem}[section]

\newtheorem{lemma}[theorem]{Lemma}

\numberwithin{equation}{section}

\theoremstyle{definition}

\theoremstyle{remark}

\newtheorem{remarks}[theorem]{Remarks}
\newtheorem*{remark*}{Remark}

 \allowdisplaybreaks

\newcommand{\1}[1]{{\mathbf 1}{\{#1\}}}
\newcommand{\2}[1]{{\mathbf 1}{(#1)}}

\newcommand{\ZP}{\mathbb{Z}_+}
\newcommand{\R}{\mathbb{R}}
\newcommand{\RP}{\mathbb{R}_+}

\newcommand{\blob}{\mkern2mu\raisebox{2pt}{\scalebox{0.4}{$\bullet$}}\mkern2mu}

\DeclareMathOperator{\E}{\mathbb{E}}
\renewcommand{\Pr}{{\mathbb P}}

\newcommand{\tra}{{\scalebox{0.6}{$\top$}}}

\newcommand{\eps}{\varepsilon}

\newcommand{\bpi}{\mbox{\boldmath${\pi}$}}
\newcommand{\ba}{\mathbf{a}}
\newcommand{\bb}{\mathbf{b}}
\newcommand{\bd}{\mathbf{d}}

\newcommand{\by}{\mathbf{y}}
\newcommand{\0}{\mathbf{0}}

\newcommand{\rc}{{\mathrm{c}}}

\newcommand{\cF}{{\mathcal{F}}}

\newcommand{\tX}{\widetilde{X}}

\makeatletter
\def\namedlabel#1#2{\begingroup  
    (#2)%
    \def\@currentlabel{#2}%
    \phantomsection\label{#1}\endgroup
}
\makeatother

\title{Non-homogeneous random walks on a half strip\\ with generalized Lamperti drifts}
\author{Chak Hei Lo \and Andrew R. Wade} 
\date{\today}

\begin{document}

\maketitle

\begin{abstract}
We study  a Markov chain on $\RP \times S$, where $\RP$ is the non-negative real numbers and $S$ is a finite set,
in which when the $\RP$-coordinate is large, the $S$-coordinate of the process is approximately Markov with stationary
distribution $\pi_i$ on $S$. If $\mu_i(x)$ is the mean drift of the $\RP$-coordinate of the process at $(x,i) \in \RP \times S$,
we study the case where $\sum_{i} \pi_i \mu_i (x) \to 0$, which is the critical
regime for the recurrence-transience phase transition. If $\mu_i(x) \to 0$ for all $i$, it is natural to study the \emph{Lamperti} 
case where $\mu_i(x) = O(1/x)$; in that case the recurrence classification is known, but we prove new results on existence and non-existence
of moments of return times. If $\mu_i (x) \to d_i$ for $d_i \neq 0$ for at least some $i$, then it is natural to
study the \emph{generalized Lamperti} case where $\mu_i (x) = d_i + O (1/x)$. By exploiting a transformation which maps
the generalized Lamperti case to the Lamperti case, we obtain a recurrence classification and existence of moments results
for the former. The generalized Lamperti case is seen to be more subtle, as the recurrence classification
depends on correlation terms between the two coordinates of the process.
\end{abstract}



\section{Introduction}
\label{introduction}

Markov processes $(X_n, \eta_n)$ on structured state-spaces $\Sigma$ contained in $\mathbb{X} \times S$ are of interest
in many applications. In this paper we are interested in the case where $X_n \in \mathbb{X} = \RP$ and 
$\eta_n \in S$  a finite set, in which case $\Sigma$ is a \emph{half strip}. Motivating applications include
\begin{itemize}
\item modulated queues~\cite{MN}, where $X_n$ represents the queue length and $\eta_n$ tracks the state of a service regime
or buffer;
\item regime-switching processes in mathematical finance, where $\eta_n$ tracks a state of the market;
\item physical processes with internal degrees of freedom~\cite{AK}, where $\eta_n$ tracks internal momentum states of a particle.
\end{itemize}
In much of the literature, $\eta_n$ is itself a Markov chain; in this case $(X_n, \eta_n)$
is known as a \emph{Markov-modulated Markov chain} or a \emph{Markov random walk}~\cite{hp}; in the contexts of strips, 
study of these models goes back to Malyshev~\cite{VM}.
The case where $\eta_n$ is Markov also includes
processes that can be represented as \emph{additive functionals of Markov chains} \cite{rogers}.
Such models pose a variety
of mathematical questions, which have been studied rather deeply over several decades using various techniques that take advantage
of the additional Markov structure, and much is now known. 
 
Much less is known when $\eta_n$ is \emph{not} Markov. 
In the present paper, following~\cite{IF,AW}, we are interested in the case where $\eta_n$ is not Markov but is, roughly speaking, approximately Markov when $X_n$ is large,
with stationary
distribution $\pi_i$ on $S$.   This 
 relaxation is necessary to probe
more intimately the recurrence/transience phase transition for these models.
If $\mu_i(x)$ is the mean drift of the $\RP$-coordinate of the process at $(x,i) \in \Sigma$,
then crucial to the asymptotic behaviour of the process
are the asymptotics of the $\mu_i$ in comparison to the $\pi_i$.
If $\mu_i (x) \to d_i \in \R$ for each $i \in S$,
then the process is transient if $\sum_i \pi_i d_i > 0$ and positive-recurrent if $\sum_i \pi_i d_i < 0$
\cite{IF,AW}. The critical case $\sum_i \pi_i d_i  =0$ is more subtle, and to investigate the
recurrence/transience phase transition it is natural, by analogy with classical
work of Lamperti on $\RP$~\cite{L1,L2}, to study the case where
$\sum_i \pi_i \mu_i (x)  = O (1/x)$. In particular, the law of the increments
is \emph{non-homogeneous} in $X_n$, which typically precludes  $\eta_n$
 from being Markovian, but admits our weaker conditions. 

The \emph{Lamperti drift} case
in which every line has $\mu_i(x) = O(1/x)$ was studied in~\cite{AW}.
The main focus of the present paper is the \emph{generalized Lamperti drift}
case where $\mu_i (x) = d_i + O (1/x)$ with $\sum_{i \in S} \pi_i d_i = 0$.

We obtain a recurrence classification for the generalized Lamperti drift case,
and in the recurrent case we obtain results on existence and non-existence of passage-time moments,
 quantifying the recurrence. We obtain these results
by use of a transformation of the process into one with \emph{Lamperti drift},
and so we establish new results on existence and non-existence of passage-time moments
in that setting first. Our method is different from that of~\cite{AW}, which relied on
an analysis of an embedded Markov chain, in that we make use of some Lyapunov functions
for the half-strip model. 

To finish this section, we present an application of our results to a simple model of
a
 \emph{correlated random walk}.
Suppose that 
  a particle performs a random walk on $\ZP$ with a short-term memory: the distribution of $X_{n+1}$
depends not only on the current position $X_n$, but also on the `direction of travel' $X_n - X_{n-1}$. Formally,
$(X_n , X_n - X_{n-1} )$ is a Markov chain on $\ZP \times S$ with $S = \{ -1 , +1 \}$,
with
\[ \Pr [ (X_{n+1},\eta_{n+1} ) = (x+j, j) \mid (X_n , \eta_n ) = (x, i ) ] = q_{ij} (x) , ~\text{for} ~ i,j\in S. \]
Then for $i \in S$,
\[ \mu_i (x) = \E [ X_{n+1} - X_n \mid (X_n, \eta_n ) = (x, i) ] = q_{i,+1} (x) - q_{i, -1} (x) .\]
The simplest model has $q_{ii} (x) = q > 1/2$ for $x \geq 1$, so the walker has a 
tendency to continue in its direction of travel. 

More generally,  suppose that for $q \in (0,1)$ and constants $c_{-1}, c_{+1} \in \R$ and $\delta>0$,
\begin{equation}
\label{e:crw}
 q_{ij} (x) = \begin{cases} 
q  + \frac{i c_i }{2x} + O (x^{-1-\delta} ) & \text{if } j =i; \\
1 - q  - \frac{i c_i }{2x} + O (x^{-1-\delta} ) & \text{if } j \neq i.
\end{cases}
\end{equation}
Here is the recurrence/transience classification for this model, which includes as the special case $q=1/2$ the recurrence
classification in Corollary~3.1 of~\cite{AW}.

\begin{theorem}
\label{thm:Example}
Consider the correlated random walk specified by~\eqref{e:crw}. Let $c = (c_{+1} + c_{-1} ) /2$.
If $c < -q$, then the walk is positive-recurrent. If $c > q$, then the walk is transient. If $|c| \le q$, then the walk is null-recurrent.
\end{theorem}

\section{Model and main results}

\subsection{Notation}
Let $S$ be a finite non-empty set, and 
let $\Sigma$ be a subset of $\RP \times S$ that is \emph{locally finite}, i.e., $\Sigma \cap ( [0,r] \times S )$ is finite for all $r \in \RP$.
Define for each $k \in S$ the \emph{line} $\Lambda_k := \{x \in \RP: (x,k) \in \Sigma\}$.
Our process of interest is as follows.
\begin{description}
\item
[\namedlabel{ass:basic}{A}]
Suppose that $(X_n , \eta_n)$, $n \in \ZP$, is a time-homogeneous, irreducible Markov chain on $\Sigma$,
a
locally finite subset of  $\RP \times S$. Suppose that for each $k \in S$ the line $\Lambda_k$ is unbounded.
\end{description}

We will typically assume that  the displacement of the $X$-coordinate has bounded $p$-moments for some $p< \infty$:
\begin{description}
\item
[\namedlabel{ass:p-moments}{B$_\textit{p}$}]
There exists a constant $C_p< \infty$ such that for all $n \in \ZP$,
\[
\E [|X_{n+1}-X_n|^p \mid X_n = x, \, \eta_n = i ] \le C_p , \text{ for all }  (x,i) \in \Sigma.
\]  
\end{description}
Most of our results will assume that~\eqref{ass:p-moments} holds for $p >2$.
For $(x,i) \in \Sigma$ and $j \in S$ define
\[
q_{ij}(x) := \Pr [ \eta_{n+1}=j \mid X_n = x, \, \eta_n = i ].
\] 
We suppose that $\eta_n$ is approximately Markov when $X_n$ is large, in the following sense:
\begin{description}
\item
[\namedlabel{ass:q-lim}{Q$_\infty$}]
Suppose that $\lim_{x \to \infty} q_{ij}(x)=q_{ij}$ exists for all $i,j \in S$, and $(q_{ij})$ is an irreducible stochastic matrix. 
\end{description}
Under assumption~\eqref{ass:q-lim} we can define a new process $\eta^*_n$, $n \in \ZP$,
 as a Markov chain on $S$ with transition probabilities given by $q_{ij}$. Since $\eta^*_n$ is irreducible and $S$ is finite, 
there exists a unique stationary distribution $\bpi =(\pi_1,\pi_2, \ldots, \pi_{|S|})^\tra$ on $S$ with 
$\pi_j>0$ for all $j \in S$ and satisfying $\pi_j= \sum_{i \in S} \pi_i q_{ij}$ for all $j \in S$.

A basic quantity is the one-step mean drift in the $X$-coordinate on line $i$:
\[
\mu_i(x) := \E[X_{n+1}-X_n \mid X_n = x, \, \eta_n = i];
\]
note that $\mu_i(x)$ is finite if~\eqref{ass:p-moments} holds for some $p \ge 1$. 
In the simplest case, we suppose that each line has an asymptotically constant drift:
\begin{description}
\item
[\namedlabel{ass:drift-constant}{D$_\text{C}$}]
For each $i \in S$ there exists $d_i \in \R$ such that $\mu_i(x) = d_i + o(1)$ as $x \to \infty$.　
\end{description}
Then we have the following result.
\begin{theorem}
\label{t:drift-constant}
Suppose that~\eqref{ass:basic} holds, and that~\eqref{ass:p-moments} holds for some $p>1$. Suppose also that~\eqref{ass:q-lim} and~\eqref{ass:drift-constant} hold.
Then the following classification applies.
\begin{itemize}
\item If $\sum_{i \in S}d_i\pi_i>0$, then $(X_n,\eta_n)$ is transient.
\item If $\sum_{i \in S}d_i\pi_i<0$, then $(X_n,\eta_n)$ is positive-recurrent.
\end{itemize}
\end{theorem}
Theorem~\ref{t:drift-constant} is a minor generalization of Theorem~2.4 of~\cite{AW}, which took $\Sigma = \ZP \times S$; the proof there readily extends
to the statement here, so we omit the proof; earlier versions of the result, assuming some additional homogeneity, are Theorem~3.1.2 of~\cite{MM} and the results of~\cite{IF,VM}.

Theorem~\ref{t:drift-constant} has nothing to say about the much more subtle case where $\sum_{i \in S} d_i \pi_i = 0$; here further
 assumptions are required to reach any conclusion. One way to achieve $\sum_{i \in S} d_i \pi_i = 0$
is to have $d_i = 0$ for all $i \in S$. In this case, 
by analogy with the classical one-dimensional
work of Lamperti~\cite{L1,L2},
the natural setting in which to probe the recurrence-transience
phase transition is that of \emph{Lamperti drift}, as studied in~\cite{AW}, which we present in Section~\ref{ss:ld}. In this setting
we give new results on existence and non-existence of moments of passage times.

The second possibility, in which $d_i \ne 0$ for some $i \in S$ but nevertheless $\sum_{i \in S} d_i \pi_i = 0$,
leads to what we call \emph{generalized Lamperti drift}, which is the main focus of this paper and is presented in Section~\ref{ss:gld}.
Here we establish a recurrence classification as well as results on passage-time moments.

For the remainder of the paper we introduce the following shorthand to simplify notation:
\[
\E_{x,i}[\,\blob\,] = \E[\, \blob \,\mid X_n=x, \, \eta_n=i].
\]

\subsection{Lamperti drift}
\label{ss:ld}

Define for $(x,i) \in \Sigma$,
\[
\sigma_i^2(x) := \E_{x,i} [ (X_{n+1}-X_n)^2 ];
\]
note that $\sigma_i^2(x)$ is finite if~\eqref{ass:p-moments} holds for some $p \ge 2$.
The case of Lamperti drift is:
\begin{description}
\item
[\namedlabel{ass:drift-lamperti}{D$_\text{L}$}]
For each $i \in S$ there exist $c_i \in \R$ and $s^2_i \in \RP$, with at least one $s^2_i$ non-zero, such that, as $x \to \infty$,
 $\mu_i(x) = \frac{c_i}{x} + o(x^{-1})$ and $\sigma_i^2(x) = s^2_i + o(1)$.
\end{description}
To obtain results at the critical point for the phase transition we will need to strengthen the assumptions~\eqref{ass:q-lim} and~\eqref{ass:drift-lamperti} by imposing additional assumptions:
\begin{description}
\item
[\namedlabel{ass:q-lim+}{Q$_\infty^+$}]
Suppose that there exists $\delta_0 \in(0,1)$ such that $\max_{i,j \in S}|q_{ij}(x)-q_{ij}|=o(x^{-\delta_0})$ as $x \to \infty$. 
\item
[\namedlabel{ass:drift-lamperti+}{D$_\text{L}^+$}]
Suppose that there exist $\delta_1 \in (0,1)$, $c_i \in \R$, and $s^2_i \in \RP$, with at least one $s^2_i$ non-zero, such that for all $i \in S$, as $x \to \infty$,
 $\mu_i(x) = \frac{c_i}{x} + o(x^{-1-\delta_1})$ and $\sigma_i^2(x) = s^2_i + o(x^{-\delta_1})$.
\end{description}

In the Lamperti drift setting, we have the following recurrence classification.

\begin{theorem} 
\label{thm:GW}
Suppose that~\eqref{ass:basic} holds, and that~\eqref{ass:p-moments} holds for some $p>2$. Suppose also that~\eqref{ass:q-lim} and~\eqref{ass:drift-lamperti} hold.
Then the following classification applies.
\begin{itemize}
\item If $\sum_{i \in S}(2c_i-s_i^2)\pi_i>0$, then $(X_n,\eta_n)$ is transient.
\item If $|\sum_{i \in S}2c_i\pi_i|<\sum_{i \in S}s_i^2\pi_i$, then $(X_n,\eta_n)$ is null-recurrent.
\item If $\sum_{i \in S}(2c_i+s_i^2)\pi_i<0$, then $(X_n,\eta_n)$ is positive-recurrent.
\end{itemize}
If, in addition,~\eqref{ass:q-lim+} and~\eqref{ass:drift-lamperti+} hold, 
then the following condition also applies (yielding an exhaustive classification):
\begin{itemize}
\item If $|\sum_{i \in S}2c_i\pi_i|=\sum_{i \in S}s_i^2\pi_i$, then $(X_n,\eta_n)$ is null-recurrent.
\end{itemize}
\end{theorem}

Theorem~\ref{thm:GW} is a slight generalization of Theorem~2.5 of~\cite{AW}, which took $\Sigma = \ZP \times S$. The proof in~\cite{AW},
which made use of Lamperti's~\cite{L1,L2} results applied to the embedded process obtained by observing the $X$-coordinate on each visit
to a reference line, extends readily to the statement here.

One way to obtain quantitative information on the nature of recurrence is to study moments of \emph{passage times}. For $x \in \RP$, 
define the stopping time $\tau_x :=\min\{n \ge 0 : X_n \le x\}$.
First we state a result that gives conditions for $\E [ \tau_x^s]$ to be finite.

\begin{theorem}
\label{thm:L2}
Suppose that~\eqref{ass:basic} holds, and that~\eqref{ass:p-moments} holds for some $p>2$. Suppose also that~\eqref{ass:q-lim} and~\eqref{ass:drift-lamperti} hold.
If  for some $\theta>0$,
\begin{equation}
\sum_{i \in S} \left[ 2c_i+(2\theta-1)s_i^2 \right] \pi_i <0, 
\label{eqn:L2}
\end{equation}
then for any $s \in \left[0,\theta \wedge \frac{p}{2} \right]$, we have $\E[\tau_x^s]<\infty$ 
for all $x$ sufficiently large.
\end{theorem}

We have the following result in the other direction.
\begin{theorem}
\label{thm:L3}
Suppose that~\eqref{ass:basic} holds, and that~\eqref{ass:p-moments} holds for some $p>2$. Suppose also that~\eqref{ass:q-lim} and~\eqref{ass:drift-lamperti} hold.
If for some $\theta \in (0, \frac{p}{2}]$, 
\begin{equation}
\sum_{i \in S} \left[ 2c_i+(2\theta-1)s_i^2 \right] \pi_i > 0, 
\label{eqn:L3}
\end{equation}
then for any $s  \geq \theta$, we have
$\E [ \tau_x^s ] =\infty$ for all sufficiently large $X_0 > x$.
\end{theorem}

In the case where $S$ is a singleton, Theorems~\ref{thm:L2} and~\ref{thm:L3}
reduce to versions of Propositions~1 and~2, respectively, of~\cite{AIM} on passage-time
moments for Markov chains on $\RP$.

\subsection{Generalized Lamperti drift}
\label{ss:gld}

Now we turn to the main topic of this paper.
This case is the most subtle, and the asymptotic properties of the process depend not only on $\mu_i(x)$
and $\sigma_i^2(x)$ but also on
the quantities
\[
\mu_{ij}(x) :=\E_{x,i} \left[(X_{n+1}-X_n) \1 {\eta_{n+1}=j} \right] ;
\]
this alerts us to the fact that correlations between the components of the increments are now crucial.
The case of generalized Lamperti drift is the following: 
\begin{description}
\item
[\namedlabel{ass:drift-gl}{D$_\text{G}$}]
 For $i, j \in S$ there exist  $d_i \in \R$, $e_i \in \R$, $d_{ij} \in \R$ and $t^2_i \in \RP$, with at least one $t^2_i$ non-zero, such that  
\begin{itemize}
\item[(a)] for all $i \in S$, $\mu_i(x) = d_i+ \frac{e_i}{x} + o(x^{-1})$ as $x \to \infty$;
\item[(b)] for all $i \in S$, $\sigma^2_i(x) = t^2_i + o(1)$ as $x \to \infty$;
\item[(c)] for all $i,j \in S$, $\mu_{ij}(x) = d_{ij} + o(1)$ as $x \to \infty$; and
\item[(d)] $\sum_{i \in S} \pi_i d_i = 0$.
\end{itemize}
\end{description}
Note that necessarily we have the relation $d_i =\sum_{j \in S} d_{ij}$.

As in the Lamperti drift case, we need to have an additional condition at the phase boundary.
\begin{description}
\item
[\namedlabel{ass:drift-gl+}{D$_\text{G}^+$}]
There exist $\delta_2 \in (0,1)$, $d_i \in \R$, $e_i \in \R$, $d_{ij} \in \R$ and $t^2_i \in \RP$, with at least one $t^2_i$ non-zero, such that  
\begin{itemize}
\item[(a)] for all $i \in S$, $\mu_i(x) = d_i+ \frac{e_i}{x} + o(x^{-1-\delta_2})$ as $x \to \infty$;
\item[(b)] for all $i \in S$, $\sigma^2_i(x) = t^2_i + o(x^{-\delta_2})$ as $x \to \infty$; and
\item[(c)] for all $i,j \in S$, $\mu_{ij}(x) = d_{ij} + o(x^{-\delta_2})$ as $x \to \infty$.
\end{itemize}
\end{description}

We also must impose refined forms of the condition~\eqref{ass:q-lim}, where now further terms come into play.
\begin{description}
\item
[\namedlabel{ass:q-lim-gl}{Q$_\text{G}$}]
 For $i, j \in S$ there exist $\gamma_{ij} \in \R$ such that $q_{ij}(x)=q_{ij}+\frac{\gamma_{ij}}{x} +o(x^{-1})$, where $(q_{ij})$ is an irreducible stochastic matrix.
\item
[\namedlabel{ass:q-lim-gl+}{Q$_\text{G}^+$}]
 There exist $\delta_3 \in (0,1)$ and $\gamma_{ij} \in \R$ such that $q_{ij}(x)=q_{ij}+\frac{\gamma_{ij}}{x} +o(x^{-1-\delta_3})$. 
\end{description}
The fact that $\sum_{j \in S}q_{ij}(x)=1$ implies, after a calculation, that $\sum_{j \in S}\gamma_{ij}=0$ for all $i \in S$.

For understanding the statement of our recurrence classification in the generalized Lamperti case, we
need the following preliminary result on solutions $\ba = (a_1, \ldots, a_{|S|})^\tra$
to the system of equations
\begin{equation}
\label{e:d-system}
d_i+\sum_{j \in S}(a_j-a_i)q_{ij}=0, \text{ for all } i \in S ;
\end{equation}
we say that a solution $\ba = (a_1, \ldots, a_{|S|})^\tra$ is \emph{unique up to translation}
if all solutions $\ba' = (a'_1, \ldots, a'_{|S|})^\tra$ have $a'_j - a_j$ constant.
\begin{lemma}
\label{lem:L4}
Let $d_i \in \R$ and $(q_{ij})$ be an irreducible stochastic matrix with stationary distribution $\pi$ on $S$. Then the following statements are equivalent.
\begin{itemize}
\item $\sum_{i \in S}d_i \pi_i=0$.
\item There exists a solution $\ba = (a_1, \ldots, a_{|S|})^\tra$ to~\eqref{e:d-system} that is unique up to translation.
\end{itemize}
\end{lemma}

Next we give our main recurrence classification for the model with generalized Lamperti drift. The criteria involve 
solutions to~\eqref{e:d-system}; as described in Lemma \ref{lem:L4} such solutions are not unique, but nevertheless the expressions in which they appear in 
Theorem~\ref{thm:L1}  are invariant under translations (see Remark~\ref{remarks}(c)), and so the statement makes sense.

\begin{theorem} 
\label{thm:L1}
Suppose that~\eqref{ass:basic} holds, and that~\eqref{ass:p-moments} holds for some $p>2$. Suppose also that~\eqref{ass:q-lim-gl} and~\eqref{ass:drift-gl} hold.
Define $\ba = (a_1, \ldots, a_{|S|})^\tra$ to be a solution to~\eqref{e:d-system}
whose existence is guaranteed by Lemma~\ref{lem:L4}. Define
\begin{equation}
\label{eq:UV}
 U := \sum_{i \in S} \left( 2 e_i + 2 \sum_{j \in S} a_j \gamma_{ij} \right) \pi_i , 
~\text{and}~
  V := \sum_{i \in S} \left( t_i^2 + 2 \sum_{j\in S} a_j d_{ij} \right) \pi_i .\end{equation}
Then the following classification applies.
\begin{itemize}
\item If $U > V$ then $(X_n,\eta_n)$ is transient.
\item If $|U| < V$ then $(X_n,\eta_n)$ is null-recurrent.
\item If $U < - V$ then $(X_n,\eta_n)$ is positive-recurrent.
\end{itemize}
If, in addition,~\eqref{ass:q-lim-gl+} and~\eqref{ass:drift-gl+} hold, 
then the following condition also applies (yielding an exhaustive classification):
\begin{itemize}
\item If $|U|= V$ then $(X_n,\eta_n)$ is null-recurrent.
\end{itemize}
\end{theorem}

As in Section~\ref{ss:ld}, we quantify the degree of recurrence by establishing 
existence and non-existence of moments of the passage times $\tau_x$. First we give
conditions for existence of moments.
\begin{theorem}
\label{thm:L4}
Suppose that~\eqref{ass:basic} holds, and that~\eqref{ass:p-moments} holds for some $p>2$. Suppose also that~\eqref{ass:q-lim-gl} and~\eqref{ass:drift-gl} hold.
Define $\ba = (a_1, \ldots, a_{|S|})^\tra$ to be a solution to~\eqref{e:d-system}
whose existence is guaranteed by Lemma~\ref{lem:L4}.
If for some $\theta >0$, with $U$ and $V$ as given by~\eqref{eq:UV},
\begin{equation}
U + (2\theta -1 ) V < 0,
%\sum_{i \in S} \biggl[ 2e_i+(2\theta-1)t_i^2 + 2\sum_{j \in S}a_j(\gamma_{ij}+(2\theta-1)d_{ij}) \biggr] \pi_i<0, \label{eqn:L4}
\end{equation}
then for any $s \in \left[0,\theta \wedge \frac{p}{2} \right]$, we have
$\E[\tau_x^s]<\infty$ for all $x$ sufficiently large.
\end{theorem}

Finally, we give conditions for non-existence of moments. 
\begin{theorem}
\label{thm:L5}
Suppose that~\eqref{ass:basic} holds, and that~\eqref{ass:p-moments} holds for some $p>2$. Suppose also that~\eqref{ass:q-lim-gl} and~\eqref{ass:drift-gl} hold.
Define $\ba = (a_1, \ldots, a_{|S|})^\tra$ to be a solution to~\eqref{e:d-system}
whose existence is guaranteed by Lemma~\ref{lem:L4}.
If for some $\theta \in (0, \frac{p}{2}]$,  with $U$ and $V$ as given by~\eqref{eq:UV},
\begin{equation}
U + ( 2 \theta - 1 ) V > 0,
%\sum_{i \in S} \biggl[ 2e_i+(2\theta-1)t_i^2 + 2\sum_{j \in S}a_j(\gamma_{ij}+(2\theta-1)d_{ij}) \biggr] \pi_i > 0, \label{eqn:L5}
\end{equation}
then for any $s \geq \theta$, we have $\E[\tau_x^s]=\infty$ for all sufficiently large $X_0 > x$.
\end{theorem}
\begin{remarks}
\label{remarks}
\begin{enumerate}[label=(\alph*),leftmargin=0pt,itemindent=20pt,nosep]
\item The generalization of the state-space $\Sigma$ from $\ZP \times S$ considered in~\cite{AW} and previous work is not merely for the sake of generalization; it is necessary
for the technical approach of this paper, whereby we find a transformation $\phi : \Sigma \to \Sigma'$ such that if $(X_n,\eta_n)$ has generalized Lamperti drift,
then $\phi (X_n, \eta_n)$ has Lamperti drift (i.e., the constant components of the drifts are eliminated). We then apply the results of Section~\ref{ss:ld}
to deduce the results in Section~\ref{ss:gld}. Even if $\Sigma = \ZP \times S$, the state-space $\Sigma'$ obtained after the transformation $\phi$ will not be
(lines are translated in a certain way).
\item The local finiteness assumption ensures that transience of the Markov chain $(X_n, \eta_n)$ is equivalent to $\lim_{n \to \infty} X_n = +\infty$, a.s., and hence
all our conditions on $\mu_i(x)$ etc.~are asymptotic conditions as $x \to \infty$.
\item As mentioned above, the non-uniqueness of solutions to~\eqref{e:d-system} is not a problem for the statement of the theorems
in this section, because the quantities in our conditions are unchanged under translation of the $a_i$. Indeed, Lemma~\ref{lem:L4} shows that 
if $(a_i,i \in S)$ is a solution then so is $(c+a_i, i \in S)$ for any $c \in \R$, and, furthermore, every solution is of this form. 
Moreover, 
the facts that $\sum_{j \in S}\gamma_{ij}=0$ and $\sum_{i \in S}\sum_{j \in S}d_{ij}\pi_i=\sum_{i \in S}d_i\pi_i=0$ guarantee that replacing every $a_i$ by $c+a_i$
does not change the quantities $U$ and $V$ given by~\eqref{eq:UV}.    
\end{enumerate}
\end{remarks}

\section{Lyapunov function estimates}

Our analysis is based on the Lyapunov function $f_\nu : \Sigma \to (0,\infty)$ defined for $\nu \in \R$ by
\begin{equation*}
f_\nu(x, i) :=
\begin{cases}
 x^\nu+\frac{\nu}{2} b_{i} x^{\nu-2} & \text{if } x \ge x_0,\\
x_0^\nu+\frac{\nu}{2} b_{i} x_0^{\nu-2} & \text{if } x < x_0,
\end{cases}
\end{equation*}
where 
$b_i \in \R$ and $x_0 := 1+ \sqrt{|\nu|  \max_{i \in S}|b_{i}| }$.

The following straightforward result, whose proof is omitted, establishes bounds on~$f_\nu$.
\begin{lemma}
\label{cal21}
Suppose $\nu \in \R$. There exist positive constants $k_1,k_2 \in (0, \infty)$,
 depending on $\nu$  and $(b_i, i \in S)$, such that 
\[
k_1 (1+x)^\nu \le f_\nu(x,i) \le k_2(1+x)^\nu, \text{ for all } x \geq 0 \text{ and all } i \in S.
\]
\end{lemma}

The next result, which is central to what follows, provides increment moment estimates for our Lyapunov function
in the Lamperti drift case.

\begin{lemma}
\label{lem:calf}
Suppose that~\eqref{ass:basic} holds, and that~\eqref{ass:p-moments} holds for some $p>2$. Suppose also that~\eqref{ass:q-lim} and~\eqref{ass:drift-lamperti} hold.
Then for any $\nu \in (2-p, p]$, we have, as $x \to \infty$,
\begin{align}
\label{e:lyapunov1}
& {} 
\E_{x,i} \left[  f_\nu ( X_{n+1}, \eta_{n+1} ) - f_\nu (X_n, \eta_n) \right] 
\nonumber\\
  &  \qquad {}
= \frac{\nu}{2} x^{\nu-2} \left( 2 c_i + (\nu-1) s_i^2 + \sum_{j \in S} (b_j - b_i) q_{ij} + o(1) \right) 
 .\end{align}
\end{lemma}

The rest of this section is devoted to the proof of Lemma~\ref{lem:calf}.
Denote $\Delta_n :=X_{n+1}-X_n$, and consider the event $E_n := \{ | \Delta_n | \leq X_n^\zeta \}$ where $\zeta \in (0,1)$.
The basic idea behind the proof of Lemma~\ref{lem:calf} is to use a Taylor's formula expansion. Such an expansion is valid only if $\Delta_n$ is not too large;
to handle various truncation estimates we will thus need the following result.

\begin{lemma}
\label{lem:indicator}
Suppose that~\eqref{ass:p-moments} holds for some $p>2$. Then for any $\zeta \in (0,1)$ and any $q \in [0,p]$, we have 
\begin{equation}
\label{col0}
\E_{x,i} \left[ |\Delta_n|^q \2 { E^\rc_n}  \right] \le C_p x^{\zeta(q-p)}. 
\end{equation}
Furthermore, if $\zeta \in (\frac{1}{p-1},1)$, we have
\begin{align}
\E_{x,i}\left[\Delta_n \2{ E_n} \right] &= \E_{x,i}\left[ \Delta_n\right] + o \left(x^{-1}\right), 
\label{col1} \\
\E_{x,i}\left[\Delta_n^2 \2{ E_n} \right] &= \E_{x,i}\left[\Delta_n^2 \right] + o \left(1 \right). 
\label{col2}
\end{align}
\end{lemma}
\begin{proof}
For $q \in [0,p]$,
\begin{align}
|\Delta_n|^q \2{ E^\rc_n}   = |\Delta_n|^p |\Delta_n|^{q-p} \2{ E^\rc_n} \le |\Delta_n|^p X_n^{\zeta(q-p)} \label{cal3}
\end{align}
The inequality follows as $q-p \le 0$, so under the condition that $|\Delta_n| > X_n^\zeta$, we have $|\Delta_n|^{q-p} \le (X_n^{\zeta})^{(q-p)} $. 
Taking the conditional expectation on both sides of~\eqref{cal3} and using the condition~\eqref{ass:p-moments}, we obtain~\eqref{col0}.
For the second statement, we use the fact that for $r \in \{1,2\}$, 
\[
\E_{x,i}\left[\Delta_n^r\right] = \E_{x,i}\left[\Delta_n^r \2{ E_n} \right] + \E_{x,i}\left[\Delta_n^r \2{ E^\rc_n} \right],
\]
where, by the $q=r$ case of~\eqref{col0},
\begin{align}
\left| \E_{x,i}\left[\Delta_n^r \2{ E^\rc_n} \right]\right| \le  \E_{x,i}\left[ |\Delta_n|^r \2{ E^\rc_n} \right] \le  C_p x^{\zeta(r-p)}. \label{cal5}
\end{align}
When $r=1$, we choose $\zeta \in (\frac{1}{p-1}, 1)$, so we have $\zeta (1-p)<-1$, and then~\eqref{cal5} gives~\eqref{col1}. When $r=2$, we know $r-p<0$, and then~\eqref{cal5} gives~\eqref{col2}.
\end{proof}

To obtain Lemma~\ref{lem:calf}, we decompose the increment of $f_\nu$. First note that, for $\zeta \in (0,1)$,
\begin{align}
\label{e:two_terms1}
 \E_{x,i}  \left[f_\nu(X_{n+1},\eta_{n+1})-f_\nu(X_n,\eta_n) \right]
 &  = \E_{x,i} \left[  ( f_\nu(X_{n+1},\eta_{n+1})-f_\nu(X_n,\eta_n) ) \2{ E_n} \right] \nonumber\\
 & \qquad {} + \E_{x,i} \left[  (f_\nu(X_{n+1},\eta_{n+1})-f_\nu(X_n,\eta_n) ) \2{ E^\rc_n} \right]  
  .
\end{align}
Choose $\zeta \in (\frac{1}{p-1},1)$ and $x_1 \in \RP$
such that $x_1 - x_1^\zeta \geq x_0$; then on the event $E_n \cap \{ X_n \geq x_1\}$ we have $X_{n+1} \geq x_1 - x_1^\zeta \geq x_0$.
Thus, for all $x \geq x_1$, we may write
\begin{align}
\label{e:two_terms2}
   & \quad {}  \E_{x,i} \left[ \left( f_\nu(X_{n+1},\eta_{n+1})-f_\nu(X_n,\eta_n)\right) \2{ E_n} \right]    \nonumber\\
 & = \E_{x,i}\left[\left(X_{n+1}^\nu-X_{n}^\nu \right) \2{ E_n} \right]  +\frac{\nu}{2} \E_{x,i}\left[\left(b_{\eta_{n+1}}X_{n+1}^{\nu-2}-b_{\eta_n}X_{n}^{\nu-2}\right)\2{ E_n} \right] .
\end{align}
We proceed to estimate the terms on the right-hand sides of~\eqref{e:two_terms1} and~\eqref{e:two_terms2} separately, via a series of lemmas. 

\begin{lemma}
\label{cal4}
Suppose that~\eqref{ass:p-moments} holds for some $p>2$. Suppose also that~\eqref{ass:drift-lamperti} holds.
Let  $\zeta \in (\frac{1}{p-1},1)$. Then for any $r \in \R$, we have, as $x \to \infty$,
\[ \E_{x,i} \left[ \left( X_{n+1}^r - X_n^r \right) \2{ E_n} \right] = r x^{r-2} \left( c_i + \frac{r-1}{2} s_i^2 + o(1) \right) .\]
\end{lemma}
\begin{proof} 
By Taylor's formula we have that
\begin{align}
 \E_{x,i}\left[\left(X_{n+1}^r-X_{n}^r \right) \2{ E_n} \right]  
& = x^r \E_{x,i} \left[ \left( ( 1 + x^{-1} \Delta_n )^r - 1 \right) \2 { E_n } \right] \nonumber\\
& = x^r \E_{x,i} \left[\left(r \left(\frac{\Delta_n}{X_{n}}\right)+\frac{r(r-1)}{2}\left(\frac{\Delta_n}{X_{n}}\right)^2 \right)\2{ E_n} +Z \right], 
\label{cal1}
\end{align}
where $|Z| \le C X_n^{-3}|\Delta_n|^3 \2{ E_n}$ for some fixed constant $C \in \RP$. To bound the term $\E_{x,i}[Z]$, first we observe that
\[
|Z| \le CX_n^{-3} |\Delta_n|^2|\Delta_n|\2{ E_n} \le C |\Delta_n|^2 X_n^{\zeta-3}. \]
 Taking expectations on both sides of the last inequality, we obtain 
\begin{align*}
\left| \E_{x,i}[Z] \right| \le \E_{x,i} |Z| \le C x^{\zeta-3} \E_{x,i}[|\Delta_n|^2] = O ( x^{\zeta-3} ),
\end{align*}
using~\eqref{ass:p-moments}. Since $\zeta < 1$ this implies  $\E_{x,i}[Z]= o(x^{-2})$, so  the expression~\eqref{cal1} becomes 
\begin{align}
& \quad {} \E_{x,i}\left[\left(X_{n+1}^r-X_{n}^r \right) \2{ E_n} \right] \nonumber \\
& = r x^{r-1}\E_{x,i}\left[ \Delta_n \2{ E_n}\right] + \frac{r (r-1)}{2}x^{r-2}\E_{x,i}\left[\Delta_n^2 \2{ E_n}\right]+o\left(x^{r-2}\right) .
\label{cal2} 
\end{align}
Then by  Lemma~\ref{lem:indicator} together with the facts that, under~\eqref{ass:drift-lamperti},
 \begin{align*}
\E_{x,i}[\Delta_n] = \mu_i(x) = \frac{c_i}{x}+o(x^{-1}), \text{ and } \E_{x,i}[\Delta_n^2] =\sigma^2_i(x) =s_i^2 + o(1),
\end{align*}
we obtain
 \begin{align*}
\E_{x,i}\left[\left(X_{n+1}^r-X_{n}^r \right) \2{ E_n} \right]  
& = r x^{r-1}\left(\frac{c_i}{x}+o\left(x^{-1}\right)\right) + \frac{r(r-1)}{2}x^{r-2}\left(s_i^2 + o(1)\right)+o\left(x^{r-2}\right) , 
\end{align*}
from which the statement in the lemma follows.
\end{proof}

\begin{lemma}
\label{cal11}
Suppose that~\eqref{ass:p-moments} holds for some $p>0$.  
Let $r \in \R$ and $\zeta \in (0,1)$, and let $g : S \to \R$. Then, as $x \to \infty$,
\[ \E_{x,i} \left[ \left( g (\eta_{n+1} ) X_{n+1}^r - g (\eta_n ) X_n^r \right) \2{ E_n} \right] = x^r \sum_{j \in S} (g(j) - g(i) ) q_{ij}(x) + o (x^r) .\]
\end{lemma}
\begin{proof} 
First observe that
\begin{align}
\label{eq1}
& \quad \E_{x,i} \left[ \left( g (\eta_{n+1} ) X_{n+1}^r - g (\eta_n ) X_n^r \right) \2{ E_n} \right] \nonumber\\
& =\E_{x,i} \left[   g (\eta_{n+1} ) \left( X_{n+1}^r -  X_n^r \right) \2{ E_n} \right] + \E_{x,i} \left[ \left( g (\eta_{n+1} )  - g (\eta_n ) \right)  X_n^r \2{ E_n} \right] .
\end{align}
We deal with the two terms on the right-hand side of~\eqref{eq1} separately. First,
\begin{align*}
\left| \E_{x,i} \left[   g (\eta_{n+1} ) \left( X_{n+1}^r -  X_n^r \right) \2{ E_n} \right] \right| \leq \Bigl( \max_{j \in S} | g(j) |  \Bigr) 
\E_{x,i} \left[  \left| X_{n+1}^r -  X_n^r \right| \2{ E_n} \right] ,
\end{align*}
where, by Taylor's formula, given $X_n =x$,
\[ \left| X_{n+1}^r -  X_n^r \right| \2{ E_n} = O ( x^{r+\zeta -1} ) = o (x^r ) .\]
On the other hand,
\begin{align*}
\E_{x,i} \left[ \left( g (\eta_{n+1} )  - g (\eta_n ) \right)  X_n^r \2{ E_n} \right] & = x^r \sum_{j \in S}  ( g (j )  - g ( i ) )  \Pr_{x,i} \left[ \{ \eta_{n+1} = j \} \cap E_n \right] .
\end{align*}
Here
$\left| \Pr_{x,i} \left[ \{ \eta_{n+1} = j \} \cap E_n \right] - q_{ij} (x) \right| \leq \Pr_{x,i} \left[ E^\rc_n \right] \to 0$, 
by the $q=0$ case of Lemma~\ref{lem:indicator}. Combining these calculations gives the result.
\end{proof}

Combining the last two results we obtain the following estimate for the first term on the right-hand side
of~\eqref{e:two_terms1}.

\begin{lemma}
\label{cal6}
Suppose that~\eqref{ass:p-moments} holds for some $p>2$. Suppose also that~\eqref{ass:drift-lamperti} and~\eqref{ass:q-lim} hold.
Let $\zeta \in (\frac{1}{p-1},1)$. Then for any $r \in \R$, we have, as $x \to \infty$,
\[ \E_{x,i} \left[ \left( f_r (X_{n+1},\eta_{n+1} ) - f_r (X_n,\eta_n) \right) \2{ E_n} \right] 
= \frac{r}{2} x^{r-2} \left( 2 c_i +(r-1) s_i^2 + \! \sum_{j \in S} (b_j - b_i ) q_{ij} + o(1) \right)
 .\]
\end{lemma}
\begin{proof}
In the equation~\eqref{e:two_terms2}
we apply Lemma~\ref{cal4} and  Lemma~\ref{cal11} with $g(y)=b_y$ and $r-2$ in place of $r$;
 together with~\eqref{ass:q-lim} we obtain the result.
\end{proof}

We have the following estimate for the second term on the right-hand side
of~\eqref{e:two_terms1}.

\begin{lemma}
\label{cal17}
Suppose that~\eqref{ass:p-moments} holds for some $p>2$. 
Then for any $r \in (2-p, p]$, we can choose $\zeta \in (0,1)$ for which, as $x \to \infty$, 
\begin{align*}
\E_{x,i} \left[ \left|  f_r ( X_{n+1}, \eta_{n+1} ) - f_r (X_n, \eta_n) \right|  \2{ E^\rc_n}  \right] = o (x ^{r-2} ).
\end{align*}
\end{lemma}
\begin{proof} 
 We may suppose throughout this proof that $X_n \ge 1$. 
First suppose that $r \in (0,p]$. 
If $|\Delta_n|\le\frac{X_n}{2}$, then $\frac{X_n}{2}\le X_n+ \Delta_n \le \frac{3X_n}{2}$. Thus with Lemma~\ref{cal21} we have, on $\{ |\Delta_n|\le\frac{X_n}{2} \}$,
\begin{align}
f_r(X_{n+1}, \eta_{n+1}) &\le k_2 \left(1+\frac{3X_n}{2}\right)^r \le C_1 \left(1+X_n \right)^r , \label{cal28}
\end{align}
for 
some constant $C_1 \in \RP$. On the other hand, if $|\Delta_n|> \frac{X_n}{2}$, then $0 \le X_{n+1} = X_n + \Delta_n  \le 3|\Delta_n|$. So with Lemma~\ref{cal21}, 
we have, on $\{ |\Delta_n| > \frac{X_n}{2} \}$,
\begin{align}
\label{cal27}
f_r(X_{n+1}, \eta_{n+1}) &\le k_2 \left(1+3|\Delta_n|\right)^r \le C_2\left|\Delta_n \right|^r , 
\end{align}
for some constant $C_2 \in \RP$. Combining the results of \eqref{cal28} and \eqref{cal27}, 
we obtain for $r >0$,
\begin{equation}
f_r(X_{n+1}, \eta_{n+1}) \le C_3 \left(1+X_n \right)^r + C_3 |\Delta_n|^r, \label{cal32}
\end{equation}
for some $C_3 \in \RP$. Hence, for $r >0$, for some $C \in \RP$, 
\begin{align*}
\left|\E_{x,i} \left[ (f_r(X_{n+1},\eta_{n+1})-f_r(X_n,\eta_n)) \2{ E^\rc_n} \right] \right| 
& \le f_r (x, i) \Pr_{x,i} \left[ E^\rc_n \right] + \E_{x,i} \left[ | f_r(X_{n+1},\eta_{n+1}) | \2{ E^\rc_n} \right]  \\
& \leq C \left(1+x \right)^r \Pr_{x,i} \left[ E^\rc_n \right] + C \E_{x,i} \left[ |\Delta_n|^r \2{ E^\rc_n} \right] ,
\end{align*}
where we have used Lemma~\ref{cal21} and inequality~\eqref{cal32}.
Then, by the $q=0$ and $q=r \in (0,p]$ cases of Lemma~\ref{lem:indicator} we have
\[ \left|\E_{x,i} \left[ (f_r(X_{n+1},\eta_{n+1})-f_r(X_n,\eta_n)) \2{ E^\rc_n} \right] \right| = O(x^{r-p\zeta})+O(x^{\zeta(r-p)}) = O(x^{r-p\zeta}) ,\]
since $\zeta<1$. This last term is $o(x^{r-2})$ provided $r-p\zeta<r-2$, i.e., $\zeta>\frac{2}{p}$.

Finally, suppose that $r \in (2-p, 0]$. Now by Lemma~\ref{cal21}, we have
$0 \leq f(x, i) \leq C$ for some $C \in \RP$ and all $x$ and $i$, 
so that
\begin{align*}
  \left| \E_{x,i} \left[\left(f_r(X_{n+1},\eta_{n+1})-f_r(X_n,\eta_n)\right) \2{ E^\rc_n} \right]\right| \le C \Pr_{x,i} \left[ E^\rc_n \right] = O(x^{-p\zeta}),
\end{align*}
by the $q=0$ case of Lemma~\ref{lem:indicator}. Since $r > 2-p$, we can choose $\zeta$ such that $0<\frac{2-r}{p}<\zeta<1$, which gives $-\zeta p < r -2$. 
\end{proof}

Now we are ready to complete the proof Lemma~\ref{lem:calf}. 

\begin{proof}[Proof of Lemma~\ref{lem:calf}]
The expression for the first moment in~\eqref{e:lyapunov1}
is simply a combination of the $r=\nu$ cases of Lemmas~\ref{cal6} and~\ref{cal17}.
\end{proof}

\section{Proofs of results for Lamperti drift}

\subsection{Some consequences of the Fredholm alternative}

  In this section, vectors are   column vectors on $\R^{|S|}$, 
 $\0$ denotes the column vector whose components are all zero, and $I$ denotes the $|S| \times |S|$ identity matrix.
We will need the following well-known algebraic result.

\begin{lemma}[Fredholm alternative]
\label{lem:fa}
Given an $|S| \times |S|$ matrix $A$ and a  column vector $\bb$, the equation $A\ba = \bb$ has a solution $\ba$ if and only if  any  column vector $\by$
for which $A^\tra \by = \0$ satisfies $\by^\tra \bb = 0$.
\end{lemma}
See \cite{AR} for other formulations and the proof of this theorem. 
First of all, we shall give the proof of Lemma~\ref{lem:L4}.
\begin{proof}[Proof of Lemma~\ref{lem:L4}]
First we write the system of equations~\eqref{e:d-system} in matrix form.
To this end, denote by $Q = (q_{ij})_{i,j\in S}$ the transition matrix for the Markov chain $\eta^*_n$ on $S$, 
and denote column vectors $\ba =(a_1, a_2, \ldots, a_{|S|})^\tra$ and $\bd =(d_1, d_2, \ldots, d_{|S|})^\tra$. Then~\eqref{e:d-system} is equivalent to
\[
(Q-I) \ba = -\bd.
\]
Setting $A=Q-I$ and $\bb = -\bd$, Lemma~\ref{lem:fa} shows that~\eqref{e:d-system} has a solution $\ba$
if and only if any  column vector $\by$ such that $(Q-I)^\tra \by = \0$ satisfies $\by^\tra \bd = 0$. But $(Q-I)^\tra \by = \0$
is equivalent to $\by^\tra Q = \by^\tra$, which implies that $\by = \alpha \bpi$ ($\alpha \in \R$) is a scalar multiple of the (unique) stationary
distribution for $Q$. Thus~\eqref{e:d-system}  
has a solution $\ba$ if and only if $\bpi^\tra \bd = 0$, i.e., $\sum_{i \in S}d_i \pi_i=0$, the special case that $\alpha=0$ contributing nothing to the condition. 

Finally, we show that any solution $\ba$ to~\eqref{e:d-system} 
  is unique up to translation. 
	Suppose there are two solutions, $\ba'$ and $\ba''$, so that $(Q-I)\ba' = (Q-I)\ba'' = -\bd$;
	 thus $(Q-I)(\ba'-\ba'')=\0$. In other words, $Q (\ba'-\ba'') =  \ba'-\ba''$. As $Q$ is a stochastic matrix, this means that $\ba'-\ba''$
	is a scalar multiple of the eigenvector $(1,1, \ldots, 1)^\tra$ corresponding to eigenvalue~$1$. Thus the components of $\ba'$ and $\ba''$
	differ by a fixed amount.
\end{proof}

A modification of the above argument yields the following statements, with inequalities instead of equality,
which will enable us to show that, under 
appropriate conditions
involving $\pi_j$, 
 suitable $b_i$ exist to construct the Lyapunov function $f_\nu$ satisfying appropriate supermartingale conditions.

\begin{lemma}
\label{lem:L2}
Let $u_i \in \R$ for each $i \in S$.
\begin{itemize}
\item[(i)]
Suppose $\sum_{i \in S} u_i \pi_i<0$.
Then there exist $(b_i, i \in S)$ such that $u_i +\sum_{j \in S} (b_j-b_i)q_{ij}<0$ for all $i$.
\item[(ii)]
Suppose   $\sum_{i \in S} u_i \pi_i >0$.
Then there exist $(b_i, i \in S)$ such that $u_i +\sum_{j \in S}(b_j-b_i)q_{ij}>0$ for all $i$.
\end{itemize}
\end{lemma}
\begin{proof}
We prove only part~(i); the proof of~(ii) is similar.
Suppose that $\sum_{i \in S} u_i \pi_i = -\eps$ for some $\eps >0$.
Then taking $\eps_i = \frac{\eps}{|S| \pi_i}$
we get $\sum_{i \in S} (u_i + \eps_i ) \pi_i = 0$.
An application of Lemma~\ref{lem:L4}
with $d_i = u_i + \eps_i$ shows that there exist $b_i$ such that
\[ u_i +\eps_i + \sum_{j \in S}(b_j-b_i)q_{ij} = 0, \text{ for all } i \in S,\]
which gives the result since $\eps_i >0$.
\end{proof}

\subsection{Proof of Theorem~\ref{thm:L2}}

To obtain existence of moments of hitting times, we apply the following semimartingale result,
which is a reformulation of Theorem~1 from~\cite{AIM}.
\begin{lemma}
\label{thm:moments}
Let $W_n$ be an integrable $\cF_n$-adapted stochastic process, taking values in an unbounded subset of $\RP$, 
with $W_0=w_0$ fixed. Suppose that there exist $\delta>0$, $w>0$, and $\gamma<1$ such that for any $n \ge 0$, 
\begin{equation}
\E[W_{n+1}-W_n \mid \cF_n] \le -\delta W_n^\gamma , \text{ on } \{n < \lambda_w\}, \label{eqn1}
\end{equation}
where $\lambda_w=\min\{n \ge 0 : W_n \le w\}$. Then, for any $s \in [0,\frac{1}{1-\gamma}]$, $\E[\lambda_w^s]< \infty$.
\end{lemma}

Now we can give the proof of Theorem~\ref{thm:L2}.

\begin{proof}[Proof of Theorem~\ref{thm:L2}]
Set $W_n := f_\nu ( X_n, \eta_n)$ for $\nu \in (0,p]$; note $W_n \to \infty$ as $X_n \to \infty$.
We aim to show that~\eqref{eqn1} holds with $\gamma= \frac{\nu-2}{\nu} < 1$, for appropriate choice of $(b_i, i \in S)$.
First note that, for $X_n$ sufficiently large,
\begin{align*}
W_n^{\gamma} & =  \left(X_n^\nu+\frac{\nu}{2} b_{\eta_n} X_n^{\nu-2}\right)^{\frac{\nu-2}{\nu}} \\
& =  X_n^{\nu-2}\left(1+\frac{\nu}{2} b_{\eta_n} X_n^{-2}\right)^{\frac{\nu-2}{\nu}} \\
& = X_n^{\nu-2} + O\left(X_n^{\nu-4}\right) ,
\end{align*}
using the fact that $b_{\eta_n}$ is uniformly bounded. In other words, $X_n^{\nu -2 } = W_n^\gamma + o ( W_n^\gamma )$,
so we have from Lemma~\ref{lem:calf} that
\begin{align}
\label{e:511a} 
 \E   [  W_{n+1} - W_n   \mid X_n , \eta_n ] 
= \frac{\nu}{2} W_n^\gamma \left( 2 c_{\eta_n} + (\nu-1) s_{\eta_n}^2 + \sum_{j \in S} (b_j - b_{\eta_n}) q_{{\eta_n}j} \right)  + o(W_n^\gamma) 
 .\end{align}
Take $\nu =  p \wedge 2\theta$ and set $u_i = 2c_i + (\nu-1) s_i^2$; then, by~\eqref{eqn:L2},
\[ \sum_{i\in  S} u_i \pi_i \leq \sum_{i\in  S} \left[ 2 c_i + (2\theta-1) s_i^2 \right] \pi_i < 0 ,\]
so that by Lemma~\ref{lem:L2}(i) we may choose $(b_i, i \in S)$ so that the coefficient of $W_n^\gamma$ on the right-hand side
of~\eqref{e:511a} is strictly negative, uniformly in $\eta_n \in S$. Hence there exists $\delta >0$ such that
\[
 \E   [  W_{n+1} - W_n   \mid X_n , \eta_n ] \le - \delta W_n^\gamma  , \text{ on } \{ W_n \geq w \},
\]
for some $w$ big enough. 
Note that $\frac{1}{1-\gamma} = \frac{\nu}{2} = \theta \wedge \frac{p}{2}$;
thus we may apply Lemma~\ref{thm:moments} to conclude that for all $w$ sufficiently large,
for any $s \in \left[0,\theta \wedge \frac{p}{2}\right]$, 
we have $\E [\lambda_w^s]<\infty$, where $\lambda_w=\min\{n \ge 0 : W_n \le w\}$. 

It remains to deduce that $\E[\tau_x^s] < \infty$ for all $x$ sufficiently large. But Lemma~\ref{lem:indicator}
shows that $X_n \leq C W_n^{1/\nu}$ for some $C \in \RP$, so $\{W_n \leq w\}$ implies that $\{X_n \leq C w^{1/\nu} \}$.
It follows that $\tau_{Cw^{1/\nu}} \leq \lambda_w$, completing the proof of the theorem.
\end{proof}

\subsection{Proof of Theorem~\ref{thm:L3}}

To obtain non-existence of moments of hitting times, we apply the following semimartingale result,
which is a variation on Theorem~2 from~\cite{AIM}.
\begin{lemma}
\label{thm:nomo}
Let $Z_n$ be a $\cF_n$-adapted stochastic process taking values in an unbounded subset of $\RP$. Suppose that there exist finite positive constants $z$, $B$, and $c$ such that, for any $n \ge 0$,
\begin{align}
\E[Z_{n+1}-Z_n \mid \mathcal{F}_n] & \ge -\frac{c}{Z_n}, \text{ on }\{Z_n \ge z\};
\label{infcond1} \\
\E[(Z_{n+1}-Z_n)^2 \mid \mathcal{F}_n] & \le B, \text{ on } \{Z_n \ge z\}.
\label{infcond2}
\end{align}
Suppose in addition that for some $p_0 > 0$, the process $Z_{n \wedge \lambda_z}^{2p_0}$ is a submartingale,
where $\lambda_z=\min\{n \ge 0 : Z_n \le z\}$. Then for any $r > p_0$, we have $\E[\lambda_z^r \mid Z_0 = z_0 ] = \infty$ for any $z_0 > z$.
\end{lemma}

We will apply this result with $Z_n := W_n^{1/\nu} = ( f_\nu (X_n , \eta_n ) )^{1/\nu}$, for appropriate choice of $(b_i, i \in S)$.
Thus we must establish some estimates on the first and second moments of the increments of $Z_n$;
this is the purpose of the next result.
\begin{lemma}
\label{lem:z1}
Suppose that~\eqref{ass:basic} holds, and that~\eqref{ass:p-moments} holds for some $p>2$. Suppose also that~\eqref{ass:q-lim} and~\eqref{ass:drift-lamperti} hold.
Then for any $\nu \in (0, p] $, we have
\begin{align*}
\E_{x,i}[Z_{n+1}-Z_n ] 
&= \frac{c_i}{x} + \frac{1}{2x} \sum_{j \in S} (b_{j}-b_i)q_{ij}  + o \left(x^{-1} \right); \text{ and } \\
\E_{x,i}[(Z_{n+1}-Z_n)^2 ] &\le B,
\end{align*}
where $B$ is a constant.
\end{lemma}
\begin{proof}
Again we define the event   $E_n := \{ | \Delta_n | \leq X_n^\zeta \}$ for $\zeta \in (0,1)$; then
\begin{align}
 \E_{x,i}[Z_{n+1}-Z_n ]  = \E_{x,i} \left[ \left(Z_{n+1}-Z_n \right) \2{ E_n}  \right] + \E_{x,i} \left[ \left( Z_{n+1}-Z_n \right) \2{ E^\rc_n}   \right]. \label{cal14}
\end{align}
For the first term on the right-hand side of~\eqref{cal14}, we first notice that for $x$ large enough  
\begin{align*}
f_\nu^{1/\nu}(x,i) = x \left(1 + \frac{\nu}{2}b_i x^{-2} \right)^{1/\nu} = x + \frac{b_i}{2x} + O(x^{-3}).
\end{align*}
Also, given $(X_n,\eta_n)=(x,i)$, on $E_n$  we have $|X_{n+1}-X_n| \le x^\zeta$ so that
$Z_{n+1} \2{E_n} = X_{n+1} + \frac{b_{\eta_{n+1}}}{2x} +o(x^{-1})$. As a result we get
\begin{align}
\E_{x,i} \left[ \left(Z_{n+1}-Z_n \right) \2{ E_n}  \right] &= \E_{x,i} \left[(X_{n+1}-X_n) \2{ E_n}  \right] + \frac{1}{2x} \E_{x,i} \left[(b_{\eta_{n+1}}-b_{\eta_n})\2{ E_n}  \right] +o(x^{-1}) \nonumber \\
&= \frac{c_i}{x} + \frac{1}{2x} \sum_{j \in S} (b_{j}-b_i)q_{ij}  + o \left(x^{-1} \right) ,
\label{cal80}
\end{align}
where in the last equality we used the $r=1$ case of Lemma~\ref{cal4} and the $r=0$ case of Lemma~\ref{cal11} for the first and second term respectively. On the other hand, to bound  
$\E_{x,i} \left[ \left(Z_{n+1}-Z_n \right) \2{ E^\rc_n} \right]$, we can just mimic  the proof of Lemma~\ref{cal17}, inserting additional
powers of $1/\nu$, to obtain
\[ 
\E_{x,i} \left[ \left(Z_{n+1}-Z_n \right) \2{ E^\rc_n} \right] = o(x^{-1}) ,
\]
which combined with~\eqref{cal80} gives the first statement in the lemma. For the second moment,
given $(X_n, \eta_n) = (x,i)$ we have
\begin{align*}
(Z_{n+1}-Z_n)^2 \2{ E_n} &\le (X_{n+1}-X_n)^2 \2{ E_n} + \frac{|X_{n+1}-X_n|}{ x}  |b_{\eta_{n+1}}-b_{\eta_n}| \2{ E_n} + O(x^{-1}) \\
&\le (X_{n+1}-X_n)^2 + o(1).
\end{align*}
 Taking expectations, we obtain 
\begin{align*}
\E_{x,i}\left[ (Z_{n+1}-Z_n)^2 \2{ E_n} \right] \le C,
\end{align*}
for some $C \in \RP$. On the other hand, we follow the proof of Lemma~\ref{cal17}, inserting powers of $2/\nu$ and $1/\nu$ respectively, to get  
\begin{align*}
\E_{x,i}\left[ (Z_{n+1}-Z_n)^2 \2{ E^\rc_n} \right] &= \E_{x,i}[(Z_{n+1}^2-Z_n^2)\2{ E^\rc_n} ] -2\E_{x,i}[Z_n(Z_{n+1}-Z_n)\2{ E^\rc_n} ] \\
&= o(1).  
\end{align*}
Combining these estimates the second statement in the lemma follows.
\end{proof}

Now we can complete the proof of Theorem~\ref{thm:L3}.

\begin{proof}[Proof of Theorem~\ref{thm:L3}.]
Take $\nu =   2 \theta$.
We will apply Lemma~\ref{thm:nomo} with $Z_n=(f_\nu(X_n, \eta_n))^{1/\nu}$ and $p_0 = \frac{\nu}{2}$;
we must verify~\eqref{infcond1} and~\eqref{infcond2}, and that $Z_{n \wedge \lambda_z}^{\nu}$ is a submartingale.
For the latter condition, it suffices to show that
\begin{equation}
\label{eq:Z-sub}
\E[ f_\nu (X_{n+1}, \eta_{n+1}) - f_\nu (X_n, \eta_n) \mid  X_n, \eta_n ] \ge 0, \text{ on } \{Z_n>z\}.
\end{equation}
By  hypothesis~\eqref{eqn:L3} and Lemma~\ref{lem:L2}(ii), writing $u_i = 2 c_i + (2 \theta -1) s_i^2$,
we may find $(b_i, i\in S)$ so that $u_i + \sum_{j \in S} ( b_j - b_i ) q_{ij} > 0$ for all $i$,
and then~\eqref{eq:Z-sub}
follows from the $\nu = 2 \theta$ case of Lemma~\ref{lem:calf}.

Of the remaining two conditions,~\eqref{infcond2} follows immediately from  the second statement in Lemma~\ref{lem:z1},
provided $\nu \leq p$, i.e., $\theta \leq \frac{p}{2}$.
From the first statement in Lemma~\ref{lem:z1}, we have that for all $x$ sufficiently large
\begin{align}
\label{infcond1a}
 \E_{x,i} [Z_{n+1}-Z_n  ]  
= \frac{1}{x} \left( c_i + \frac{1}{2} \sum_{j \in S} (b_{j}-b_i)q_{ij} + o(1) \right)   
 \ge - \frac{C_1}{x}   ,
\end{align}
for all $i$ and all $x$ sufficiently large, where $C_1 \in \RP$ depends on the $c_i$ and $b_i$. Now
Lemma~\ref{cal21} implies that $Z_n \leq C_2 X_n$ for some $C_2 >0$, so we deduce condition~\eqref{infcond1}
from~\eqref{infcond1a}.

Thus Lemma~\ref{thm:nomo} shows that $\E [ \lambda_z^r] = \infty$ for $X_0$ sufficiently large and $r > \theta$.
But if~\eqref{eqn:L3} holds for some $\theta \in (0, \frac{p}{2}]$, then by continuity~\eqref{eqn:L3} also holds
for some $\theta' \in (0, \theta)$, so $\E [ \lambda_z^r] = \infty$ for $r = \theta$ too.
\end{proof}

\section{Proofs of results for generalized Lamperti drift}

\subsection{Transformation to Lamperti drift case}

The idea behind the proofs of the results in Section~\ref{ss:gld} is to transform the process $(X_n, \eta_n)$ with generalized Lamperti drift
to a process $(\tX_n, \eta_n ) = (X_n +a_{\eta_n} , \eta_n)$, with appropriate choices of the real numbers $a_i$, $i \in S$, that has Lamperti drift, i.e., the constant
components of the drifts are eliminated; then we can apply the results in Section~\ref{ss:ld}, once we have at hand increment moment estimates for the
transformed process $(\tX_n, \eta_n)$.

For $d_i$, $i \in S$ with $\sum_{i \in S} \pi_i d_i = 0$ as specified under assumption~\eqref{ass:drift-gl}, 
choose $a_i$, $i \in S$ as guaranteed by Lemma~\ref{lem:L4}, so that $d_i+\sum_{j \in S}(a_j-a_i)q_{ij}=0$;
since we are free to translate the $a_i$ by any constant, we may (and do) suppose here that $a_i \geq 0$ for all $i \in S$.

Let $\Sigma' = \{ (x+a_i, i) : (x,i) \in \Sigma \}$ denote the state space of the transformed process; then $\Sigma'$ is a locally
finite subset of $\RP \times S$ with unbounded lines $\Lambda'_k = \{ x \in \RP : (x,k) \in \Sigma' \}$. The map $(x,i) \mapsto 
(x+a_i , i)$ is a bijection, so the Markov chain $(\tX_n, \eta_n )$ has precisely the same abstract structure as the Markov chain $(X_n, \eta_n)$, in particular,
the transformed process is (positive-)recurrent if and only if the original process is (positive-)recurrent, and so on.
Thus to obtain results for the process $(X_n, \eta_n)$ it is sufficient to prove results for the transformed process $(\tX_n, \eta_n )$.

For the increment moments of the transformed process, we use the notation
\begin{align*}
\widetilde{\mu}_i(y) & := \E[\widetilde{X}_{n+1}-\widetilde{X}_n \mid \widetilde{X}_n=y,\eta_n=i], \\
\widetilde{\sigma}_i^2(y) & := \E[(\widetilde{X}_{n+1}-\widetilde{X}_n)^2 \mid \widetilde{X}_n=y,\eta_n=i].
\end{align*}
\begin{lemma}
\label{lem:transform}
Suppose that~\eqref{ass:basic} holds, and that~\eqref{ass:p-moments} holds for some $p>2$. Suppose also that~\eqref{ass:q-lim-gl} and~\eqref{ass:drift-gl} hold.
Define $a_i$, $i \in S$ to be a solution to~\eqref{e:d-system} with $a_i \geq 0$ for all $i$,
whose existence is guaranteed by Lemma~\ref{lem:L4}.
Either (i) set $\delta_4 = 0$; or (ii) suppose that~\eqref{ass:q-lim-gl+} and~\eqref{ass:drift-gl+} hold and set $\delta_4 = \delta_2 \wedge \delta_3 \in (0,1)$. 
For $i \in S$, define 
\begin{equation}
\label{e:new_constants}
  c_i := e_i+\sum_{j \in S}a_j\gamma_{ij}, ~~\text{and}~~ s_i^2 := t_i^2 +2\sum_{j \in S}a_jd_{ij} + \sum_{j \in S}(a_j^2-a_i^2)q_{ij}. \end{equation}
Then we have that, as $x \to \infty$,
\begin{align*}
\widetilde{\mu}_i(x)  = \frac{c_i}{x} +o(x^{-1-\delta_4}), ~~\text{and}~~ \widetilde{\sigma}_i^2(x)  = s_i^2 + o(x^{-\delta_4}) .
\end{align*}
\end{lemma}
\begin{proof}
For concreteness, we give the proof when~\eqref{ass:q-lim-gl+} and~\eqref{ass:drift-gl+} hold; in the other
case the argument is the same but with $\delta_4$ set to $0$ throughout.
For the first moment, we have
\begin{align*}
\widetilde{\mu}_i(x)
 &= \E_{x-a_i,i}[X_{n+1}-X_n]+\sum_{j \in S}\E_{x-a_i,i}\left[\left(a_{\eta_{n+1}}-a_{\eta_n}\right) \1 {a_{\eta_{n+1}}=j }\right] \\
&= \mu_i(x-a_i)+\sum_{j \in S}(a_j-a_i)q_{ij}(x-a_i).
\end{align*}
Now using hypothesis~(a) in~\eqref{ass:drift-gl+} and~\eqref{ass:q-lim-gl+} we obtain
\begin{align*}
\widetilde{\mu}_i(x) &= d_i + \frac{e_i}{x-a_i}+\sum_{j \in S}(a_j-a_i)q_{ij}+\sum_{j \in S}(a_j-a_i)\frac{\gamma_{ij}}{x-a_i} +o((x-a_i)^{-1-\delta_4}) \\
&= d_i +\sum_{j \in S}(a_j-a_i)q_{ij} + \frac{e_i}{x} + \frac{1}{x} \sum_{j \in S} a_j\gamma_{ij} +o(x^{-1-\delta_4}),
\end{align*}
since $\sum_{j \in S}\gamma_{ij}=0$. By hypothesis~(d)  in~\eqref{ass:drift-gl+} and choice of the $a_i$
(cf Lemma~\ref{lem:L4}), the constant term here vanishes, so we obtain the
expression for $\widetilde{\mu}_i(x)$ in the lemma.

For the second moment, we have
\begin{align*}
\widetilde{\sigma}_i^2(x) 
&= \E_{x-a_i,i} \left[(X_{n+1}-X_n)^2 \right]+2\E_{x-a_i,i}\left[ (a_{\eta_{n+1}}-a_{\eta_n})(X_{n+1}-X_n)\right]+\E_{x-a_i,i}\left[a_{\eta_{n+1}}^2-a_{\eta_n}^2 \right] \\
&= s_i^2 + 2\sum_{j \in S}(a_j-a_i)\mu_{ij}(x-a_i)+\sum_{j \in S}(a_j-a_i)^2 q_{ij}(x-a_i) + o ( x^{-\delta_4} ) . \end{align*}
Then using hypothesis~(c) in~\eqref{ass:drift-gl+} and~\eqref{ass:q-lim-gl+} we obtain
\begin{align*}  
\widetilde{\sigma}_i^2(x) & = s_i^2 + 2\sum_{j \in S}(a_j-a_i)d_{ij} + \sum_{j \in S}(a_j-a_i)^2 q_{ij} +o(x^{-\delta_4})  
\\
&= s_i^2 + 2\sum_{j \in S}a_j d_{ij} + \sum_{j \in S}(a_j^2-a_i^2) q_{ij} -2a_i \biggl( d_i + \sum_{j \in S}(a_j-a_i) q_{ij} \biggr)   +o(x^{-\delta_4}),
\end{align*}
which gives the result after once again using the fact that $d_i+\sum_{j \in S}(a_j-a_i) q_{ij}=0$.
\end{proof}

\subsection{Proofs of Theorems~\ref{thm:L1}, \ref{thm:L4} and~\ref{thm:L5}}

Armed with our transformation of the process, we can now use the results in Section~\ref{ss:ld} to complete the proofs of the
theorems in Section~\ref{ss:gld}.

\begin{proof}[Proof of Theorem~\ref{thm:L1}]
Lemma~\ref{lem:transform} shows that if $(X_n, \eta_n)$ satisfies the conditions of Theorem~\ref{thm:L1},
then $(\tX_n, \eta_n)$ satisfies the conditions of Theorem~\ref{thm:GW}
with $c_i$ and $s_i^2$ as given by~\eqref{e:new_constants}.
Theorem~\ref{thm:GW} shows that the process is transient if
\begin{align*}
0 < \sum_{i \in S}[2c_i-s_i^2]\pi_i =& \sum_{i \in S}\left[2e_i + 2\sum_{j \in S}a_j\gamma_{ij}-\left(t_i^2 +2\sum_{j \in S}a_jd_{ij} + \sum_{j \in S}(a_j^2-a_i^2)q_{ij}\right)\right]\pi_i \\
=& \sum_{i \in S}\left(2e_i -t_i^2 + 2\sum_{j \in S}a_j(\gamma_{ij}-d_{ij})\right)\pi_i -\sum_{i \in S}\sum_{j \in S}(a_j^2-a_i^2)q_{ij}\pi_i ,
\end{align*}
using the expressions at~\eqref{e:new_constants}. Note that the final term here vanishes, because
\begin{align*}
\sum_{i \in S}\sum_{j \in S}(a_j^2-a_i^2)q_{ij}\pi_i 
&= \sum_{j \in S}a_j^2\sum_{i \in S}q_{ij}\pi_i - \sum_{i \in S}a_i^2\pi_i\sum_{j \in S}q_{ij} \\
&= \sum_{j \in S}a_j^2\pi_j - \sum_{i \in S}a_i^2\pi_i =0,
\end{align*}
using the fact that $\bpi$ is the stationary distribution for $(q_{ij})$. This gives the condition for transience stated in Theorem~\ref{thm:L1}.

Similarly,  the condition for positive-recurrence is
\begin{align*}
0 > \sum_{i \in S}[2c_i + s_i^2]\pi_i = & \sum_{i \in S}\left[2e_i + 2\sum_{j \in S}a_j\gamma_{ij}+\left(t_i^2 +2\sum_{j \in S}a_jd_{ij} + \sum_{j \in S}(a_j^2-a_i^2)q_{ij}\right)\right]\pi_i \\
=& \sum_{i \in S}\left(2e_i +t_i^2 + 2\sum_{j \in S}a_j(\gamma_{ij}+d_{ij})\right)\pi_i +\sum_{i \in S}\sum_{j \in S}(a_j^2-a_i^2)q_{ij}\pi_i \\
=& \sum_{i \in S}\left(2e_i +t_i^2 + 2\sum_{j \in S}a_j(\gamma_{ij}+d_{ij})\right)\pi_i ,
\end{align*}
which gives the condition for positive-recurrence in the theorem.
The case for null-recurrence and at the critical point follows accordingly by the same calculation.
\end{proof}

\begin{proof}[Proof of Theorem~\ref{thm:L4}]
The proof is analogous to the proof of Theorem~\ref{thm:L1}, this time applying Theorem~\ref{thm:L2} to the transformed process.
\end{proof}

\begin{proof}[Proof of Theorem~\ref{thm:L5}]
This time we apply Theorem~\ref{thm:L3} to the transformed process.
\end{proof}

\section{Appendix: Proof of Theorem~\ref{thm:Example}}
In this last section we complete the proof of Theorem~\ref{thm:Example} on  correlated random walk given in the introduction.

\begin{proof}[Proof of Theorem~\ref{thm:Example}]
Note first that $q_{ii} = q$ and $q_{ij} = 1-q$ for $j \neq i$; hence $\pi = (\frac12,\frac12)$.
By direct calculation, we get 
$\mu_{i}(x)=i (2q-1) + \frac{c_i}{x} + O (x^{-1-\delta})$. This gives $d_{i}= i (2q-1)$ and $e_{i}=c_i$. Now we want to solve the system of equations~\eqref{e:d-system} for $a_{i}$,
which amounts to 
\[
1-2q+(a_{+1}-a_{-1})(1-q)=0.
\]
Since the solution $(a_i)$ is unique up to translation, without loss of generality we can choose $a_{-1}=0$ and
then we get $a_{+1} =\frac{2q-1}{1-q}$. Next we observe that $d_{ii}=q$,
while if $i \neq j$ we have  $d_{ij}=1-q$; also, 
 $\gamma_{ij}= \frac{j c_i}{2}$ and $t_i^2=1$. 
Now we calculate 
\begin{align*}
 \sum_{i \in S}\biggl( 2e_i+2\sum_{j \in S}a_j \gamma_{ij} \biggr) \pi_i & = e_{+1} + e_{-1} +a_{+1} \gamma_{+1,+1} + a_{+1} \gamma_{-1,+1}    = \frac{c_{+1} + c_{-1}}{2(1-q)} ; \text{ and} \\
\sum_{i \in S} \biggl( t_i^2+2\sum_{j \in S}a_j d_{ij} \biggr) \pi_i & = 1 +a_{+1} d_{+1,+1} + a_{+1} d_{-1,+1} =\frac{q}{1-q}. 
\end{align*}
Then applying Theorem~\ref{thm:L1}  we obtain the desired result.
\end{proof}

\section*{Acknowledgements}

The authors are grateful to Nicholas Georgiou and Mikhail Menshikov for fruitful discussions on the topic of this paper.

\end{document}